\documentclass[12pt, reqno]{amsart}
\usepackage{hyperref}
\hypersetup{
	colorlinks=true,
	citecolor=blue,
	linkcolor=blue,
	filecolor=magenta,      
	urlcolor=cyan,
}

\usepackage{adjustbox}

\usepackage[margin=1.2in]{geometry}
\usepackage{setspace}
\usepackage[capitalize,noabbrev,nameinlink]{cleveref}
\usepackage{amsmath, amsthm, amsfonts, amssymb, graphicx, dsfont, stmaryrd, extarrows, enumitem, mathtools}
\usepackage[T1]{fontenc}
\usepackage{textcomp}
\usepackage{lmodern}
\usepackage{microtype} 
\usepackage{amscd}
\usepackage{mathrsfs}
\usepackage{tikz-cd}
\usetikzlibrary{matrix,patterns}
\usepackage[colorinlistoftodos]{todonotes}
\usepackage{color}
\usepackage{bbm}
\usepackage{verbatim}
\usepackage[mathscr]{euscript}

\numberwithin{equation}{section}
\theoremstyle{plain}
\newtheorem{thm}[equation]{Theorem}
\newtheorem*{thm*}{Theorem}

\newtheorem{cor}[equation]{Corollary}       
\newtheorem{lem}[equation]{Lemma}

\theoremstyle{definition}

\newtheorem{ex}[equation]{Example}

\newtheorem{rem}[equation]{Remark}

\newtheorem{chunk}[equation]{}
\crefformat{chunk}{(#2#1#3)}

\newtheoremstyle{remboldstyle}
 {}{}{}{}{\bfseries}{.}{.5em}{{\thmname{#1 }}{\thmnumber{#2}}{\thmnote{ #3}}}
\theoremstyle{remboldstyle}
\newtheorem{slab}[equation]{}
\crefformat{slab}{(#2#1#3)}

\newtheoremstyle{nonemboldstyle}
 {}{}{}{}{\bfseries}{.}{.5em}{{\thmnote{#3}}}
\theoremstyle{nonemboldstyle}

\newcommand{\Hom}{\mathrm{Hom}}
\newcommand{\ihom}{\underline{\msf{hom}}}

\newcommand{\msf}[1]{\mathsf{#1}}

\newcommand{\mc}[1]{\mathcal{#1}}
\newcommand{\mrm}[1]{\mathrm{#1}}
\newcommand{\mbb}[1]{\mathbb{#1}}
\newcommand{\scr}[1]{\mathscr{#1}}

\newcommand{\p}{\mathfrak{p}}

\newcommand{\T}{\mathsf{T}}
\newcommand{\U}{\mathsf{U}}

\newcommand{\K}{\scr{K}}

\newcommand{\1}{\mathds{1}}

\newcommand{\Spc}{\msf{Spc}}
\newcommand{\Spch}{\msf{Spc}^{\msf{h}}}
\newcommand{\supp}{\msf{supp}}

\renewcommand{\L}{\scr{L}}

\renewcommand{\mod}[1]{\mathrm{Mod}_{#1}}

\newcommand{\C}{\mathsf{C}}
\newcommand{\D}{\mathsf{D}}

\renewcommand{\mod}[1]{\msf{mod}(#1)}
\newcommand{\Mod}[1]{\msf{Mod}(#1)}
\newcommand{\Flat}[1]{\msf{Flat}(#1)}
\newcommand{\ab}{\msf{Ab}}

\newcommand{\rlim}{\varinjlim}

\renewcommand{\t}{\text}
\renewcommand{\c}{\mrm{c}}

\newcommand{\op}{\mrm{op}}

\renewcommand{\Flat}[1]{\msf{Flat}(#1)}

\newcommand{\closed}[1]{\msf{KZg}^{\otimes}_{\msf{Cl}}(#1)}
\newcommand{\y}{\msf{y}}
\renewcommand{\bar}{\overline}
\newcommand{\hspec}[1]{\msf{Spc}^{\msf{h}}(#1)}

\renewcommand{\phi}{\varphi}

\newcommand{\Irr}[1]{\msf{IrrRk}(\T^\c)}

\usepackage{todonotes}

\newcommand{\zgk}{\msf{Zg}(\K)}
\newcommand{\yy}{\mathbbm{y}}
\usetikzlibrary{babel}

\begin{document}
\title{Definable functoriality of tensor-triangular spectra}
\begin{abstract}
We prove that the homological and Balmer spectra in tensor-triangular geometry are functorial in certain definable functors, thereby providing an alternative perspective on functoriality in tensor-triangular geometry from the viewpoint of purity, and generalising current results in the literature.
\end{abstract}

\author{Isaac Bird}
\address[Bird]{Department of Algebra, Faculty of Mathematics and Physics, Charles University in Prague, Sokolovsk\'{a} 83, 186 75 Praha, Czech Republic}
\email{bird@karlin.mff.cuni.cz}
\author{Jordan Williamson}
\address[Williamson]{Department of Algebra, Faculty of Mathematics and Physics, Charles University in Prague, Sokolovsk\'{a} 83, 186 75 Praha, Czech Republic}
\email{williamson@karlin.mff.cuni.cz}
\maketitle

\section{Introduction}

Tensor-triangular (``tt'') geometry is the study of tensor-triangulated categories $\K$ from a geometric perspective via associated topological spaces, in analogy to the study of commutative rings through their Zariski spectra. There are two such associated spaces: the Balmer spectrum $\msf{Spc}(\K)$~\cite{BalmerSpec} and the homological spectrum $\Spch(\K)$~\cite{BalmerNil}. The former provides the universal solution to the problem of classifying (radical) thick $\otimes$-ideals of $\K$, whereas the latter parametrises abelian counterparts of residue fields of $\K$ and provides a nilpotence theorem. As such, understanding the Balmer and homological spectra of tt-categories is of interest across many fields: commutative algebra, representation theory, algebraic geometry, topology, and so on; see~\cite{Balmerttg} for some highlights. 

A standard yet powerful technique in tt-geometry is the use of descent~\cite{BCHSsurj, bhsz}. This relies upon the fact that the Balmer spectrum and the homological spectrum are functorial: suitable functors $\K \to \L$ induce continuous maps on the associated spectra. One can then use such maps to construct new points by comparison to well-understood examples, for instance by passage to residue fields.

In this short note, we provide a new perspective on the functoriality of the homological and Balmer spectra by utilising techniques from model theory. We show that our result generalises and recovers Balmer's result that the homological spectrum is functorial in geometric functors. Using that the Balmer spectrum is the Kolmogorov quotient of the homological spectrum~\cite{BHSComp}, we then deduce corresponding functoriality for the Balmer spectrum. Furthermore, throughout the paper we work with essentially small rigid tt-categories rather than with the compact objects in some ambient `big' category. As such, we reformulate some key features of the homological spectrum in this small setting. 

For ease of exposition in this introduction, we assume that we are in the world of rigidly-compactly generated tensor-triangulated categories. In \cite{BWHomSp}, it was illustrated how one can recover the homological spectrum (and thus Balmer spectrum) of a rigidly-compactly generated tt-category from its pure structure. This pure structure consists of precisely the triangles (resp., objects) which become short exact sequences (resp., injective) when viewed as cohomological functors, and as such reflects the abelian structure of the functor category in the triangulated setting. 

This connection between tt-geometry and model theory motivates the question of whether functors which preserve the pure structure of the categories, the \emph{definable functors} \cite{BWdef}, can also be used to provide functoriality of the homological spectrum. In this note we show that this indeed occurs. 

\begin{thm*}[{\ref{thm:functoriality}, \ref{thm:Balmerspectrumfunctoriality}}]
Let $F\colon \T \to \U$ be a definable functor between rigidly-compactly generated tt-categories. If the induced adjunction 
\[
\begin{tikzcd}
\Mod{\T^\c} \arrow[r, shift right = 0.5ex,swap,  "\bar{F}"] \arrow[r, leftarrow, shift left = 0.5ex, "\Lambda"]& \Mod{\U^\c}
\end{tikzcd}
\]
satisfies the projection formula and $\Lambda$ preserves cohomological functors, then there is a continuous map $\Spch(F)\colon\hspec{\T^\c}\to\hspec{\U^\c}$ given by
\[
\mc{B}\mapsto \Lambda^{-1}(\mc{B})\cap\mod{\U^\c}.
\]
Thus, by taking Kolmogorov quotients, $F$ induces a continuous map $\Spc(\T^{\c})\to\Spc(\U^{\c})$.
\end{thm*}
We note that the above theorem actually holds in more generality: one may replace the rigidly-compactly generated tt-categories by the categories of cohomological functors on essentially small rigid tt-categories, see \cref{setup} for more details. We also remark that one can describe the induced map $\Spch(F)$ in terms of the original functor $F$ rather than in terms of $\Lambda$, c.f., \cref{thm:functoriality} and \cref{cor:alternativeform}.

The above theorem generalises the established functoriality of Balmer \cite[Theorem 5.10]{Balmerhomsupp}, since the right adjoint to a geometric functor is a definable functor which satisfies the conditions of the above theorem, see \cref{cor:generalise}. In particular, this motivates why the above theorem provides a \emph{co}variant functoriality in definable functors, in contrast to the usual \emph{contra}variant functoriality in geometric functors.

We envisage that this perspective to functoriality through purity may be useful in other settings such as that of non-rigid tt-categories.

\subsection*{Acknowledgements}
Both authors were supported by the project PRIMUS/23/SCI/006 from Charles University, and by Charles University Research Centre program No. UNCE/24/SCI/022.

\section{Preliminaries}
\begin{slab}[Rigid tt-categories and modules]\label{slab:modules}
Let $\K$ be an essentially small rigid tt-category: in short, this means that $\K$ is a triangulated category with a compatible closed symmetric monoidal structure $(\otimes, \ihom, \1)$ for which every object is strongly dualisable (i.e., the natural map $\ihom(k,\1) \otimes k' \to \ihom(k,k')$ is an isomorphism for all $k,k' \in \K$). We let $\Mod{\K}$ denote the category of additive functors $\K^{\op}\to\ab$, which is a locally coherent Grothendieck abelian category. We write $\yy\colon \K \to \Mod{\K}$ for the Yoneda embedding given by $\yy k \coloneqq \Hom_\K(-,k)$. The full abelian subcategory of finitely presented objects in $\Mod{\K}$ is denoted by $\mod{\K}$, and consists of the functors $f$ which have a presentation
\[
\yy k \to \yy k' \to f\to 0.
\]
The Yoneda embedding gives an additive equivalence $\yy\colon\K\xrightarrow{\sim} \msf{proj}(\K)$, where $\msf{proj}(\K)$ denotes the finitely presented projective objects in $\Mod{\K}$. We write $\Flat{\K}$ for the subcategory of $\Mod{\K}$ consisting of the cohomological functors; the terminology is justified by the fact that $\Flat{\K}=\msf{Ind}(\K)=\rlim\msf{proj}(\K)$ where $\rlim$ denotes the closure under direct limits. We note that any injective object in $\Mod{\K}$ is flat by \cite[Proposition 7.1]{BelFreyd} since $\K$ is a coherent category in the sense of \cite[\S 4]{BelFreyd}.

As $\K$ is closed symmetric monoidal, $\Mod{\K}$ inherits a closed symmetric monoidal structure via Day convolution, see~\cite[Remark 2.4]{BKSrum}. The monoidal product on $\Mod{\K}$ is the unique functor $-\otimes-\colon \Mod{\K}\times \Mod{\K}\to\Mod{\K}$ which commutes with colimits in both variables, and for which $\yy\colon \K \to \Mod{\K}$ is strong symmetric monoidal. We again write $\ihom$ for the internal hom on $\Mod{\K}$ and also write $\mbb{D} = \ihom(-,\yy\1)$ for the functional dual.
\end{slab}

\begin{slab}[Purity]
Let $\C$ be a finitely accessible category with products (for example, $\Mod{\K}$ or $\Flat{\K}$).
A short exact sequence $0\to L\to M\to N\to 0$ in $\C$ is called \emph{pure} if the induced sequence \[0\to \Hom_{\C}(f,L)\to \Hom_{\C}(f,M)\to \Hom_{\C}(f,N)\to 0\] is exact for all $f\in\msf{fp}(\C)$, in which case $L\to M$ is called a \emph{pure monomorphism}, and $L$ a \emph{pure subobject} of $M$. The terms \emph{pure epimorphism} and \emph{pure quotient} are defined correspondingly for $M\to N$. An object $X\in\C$ is \emph{pure injective} if any pure monomorphism $X\to Y$ splits.

A full subcategory $\mc{D}\subseteq\C$ is \emph{definable} if it is closed under pure subobjects, direct limits, and products. The set $\msf{pinj}(\C)$ of isomorphism classes of indecomposable pure injective objects in $\C$ underlies a topological space called the \emph{Ziegler spectrum} of $\C$, denoted $\msf{Zg}(\C)$. The closed sets are given by $\mc{D}\cap\msf{pinj}(\C)$ where $\mc{D}$ is a definable subcategory of $\C$. If $\mc{X}\subseteq\C$, we let $\msf{Def}(\mc{X})$ denote the smallest definable subcategory of $\C$ containing $\mc{X}$.

The category $\Flat{\K}$ is a definable subcategory of $\Mod{\K}$, and for brevity we write $\zgk$ for the Ziegler spectrum of $\Flat{\K}$.
\end{slab}

\begin{slab}[Definable functors]
Given two finitely accessible categories $\C,\D$ with products, a functor $F\colon\C\to\D$ is \emph{definable} if it preserves direct limits and products. Definable functors preserve pure exact sequences and pure injective objects, see \cite[\S 13]{dac}. For the purpose of this paper, the key example is that for any $f \in \mod{\K}$, the functor $f \otimes -\colon \Mod{\K} \to \Mod{\K}$ is definable. This is easily seen by reduction to the case $f = \yy k$ for $k \in \K$.
\end{slab}

\begin{slab}[The fundamental correspondence]\label{fundamentalcorrespondence}
As $\Flat{\K}$ is a finitely accessible category, we can consider the category $\Mod{\msf{proj}(\K)}$ of additive functors $\msf{proj}(\K)^\op \to \ab$, and the restricted Yoneda functor $\y\colon\Flat{\K}\to\Mod{\msf{proj}(\K)}$ given by $\y\colon X\mapsto \Hom_{\Mod{\K}}(-,X)\vert_{\msf{proj}(\K)}$. 

There is, see \cite[\S 8]{dac} or \cite[Theorem 12.2.2]{Krbook}, an order reversing bijection between definable subcategories $\mc{D}\subseteq\Flat{\K}$ (equivalently, closed subsets of $\zgk$) and Serre subcategories $\mc{S}\subseteq \mod{\msf{proj}(\K)}$, which is given by the assignments
\[
\mc{D}\mapsto\{f \in \mod{\msf{proj}(\K)}:\Hom(f,\y X)=0 \t{ for all }X\in\mc{D}\}
\]
and 
\[
\mc{S}\mapsto \{X\in\Flat{\K}:\Hom(f,\y X)=0 \t{ for all }f\in\mc{S}\}.
\]
As $\yy\colon\K\to\msf{proj}(\K)$ is an equivalence, the restriction functor $\mod{\msf{proj}(\K)}\to\mod{\K}$, given by $F\mapsto F\circ \yy$, is also an equivalence, and therefore, combining this with the above bijection, we obtain an order reversing bijection between definable subcategories $\mc{D}$ of $\Flat{\K}$ and Serre subcategories $\mc{S}$ of $\mod{\K}$. These are explicitly given by 
\[
\mc{D}\mapsto \scr{S}(\mc{D})\coloneqq\{f\in\mod{\K}:\Hom(f,X)=0 \t{ for all }X\in\mc{D}\}
\]
and
\[
\mc{S}\mapsto \scr{D}(\mc{S})\coloneqq\{X\in\Flat{\K}:\Hom(f,X)=0 \t{ for all }f\in\mc{S}\}.
\]
\end{slab}

\begin{slab}[The monoidal fundamental correspondence]\label{slab:monoidalfundamental}
A Serre subcategory $\mc{S}\subseteq\mod{\K}$ is a \emph{Serre $\otimes$-ideal} if for any $f \in \mc{S}$ and $g\in\mod{\K}$ one has $f\otimes g\in\mc{S}$. A definable subcategory $\mc{D}\subseteq \Flat{\K}$ is \emph{$\otimes$-closed} if for any $X\in\mc{D}$ and $Y\in\Flat{\K}$ one has $X\otimes Y\in\mc{D}$. Note that $\mc{D}$ is a $\otimes$-closed definable subcategory if and only if $X\otimes \yy k\in\mc{D}$ for all $k \in \K$ since $\mc{D}$ is closed under direct limits. Given $\mc{X} \subseteq \Flat{\K}$, we write $\msf{Def}^\otimes(\mc{X})$ for the smallest $\otimes$-closed definable subcategory of $\Flat{\K}$ containing $\mc{X}$. One readily checks that 
\[\msf{Def}^\otimes(\mc{X}) = \msf{Def}(\yy k \otimes X : k \in \K, X \in \mc{X}).\]
Since $\scr{K}$ is rigid, the bijection of \cref{fundamentalcorrespondence} given by $\scr{S}(-)$ and $\scr{D}(-)$ restricts to give a bijection between Serre $\otimes$-ideals of $\mod{\K}$ and $\otimes$-closed definable subcategories of $\Flat{\K}$. For example, if $\mc{S}\subseteq\mod{\K}$ is a Serre $\otimes$-ideal, then $\scr{D}(\mc{S})$ is a $\otimes$-closed definable subcategory since for $X \in \scr{D}(\mc{S})$ and $f \in \mc{S}$ we have
\[
\Hom_{\Mod{\K}}(f,\yy k\otimes X) \simeq \Hom_{\Mod{\K}}(f\otimes \mbb{D}(\yy k), X)=0
\]
as $f\otimes \mbb{D}(\yy k)\in\scr{S}$. The other direction is similar.
\end{slab}

\begin{rem}
Without an assumption of rigidity, one still obtains a bijection between Serre $\otimes$-ideals of $\mod{\K}$ and definable coideals of $\Flat{\K}$, that is, definable subcategories $\mc{D}$ such that if $X\in\mc{D}$ and $k\in\K$ one has $\ihom(\yy k,X)\in\mc{D}$. However, it is no longer clear whether there is any assignment between $\otimes$-closed definable subcategories of $\Flat{\K}$ and a type of Serre subcategory of $\mod{\K}$ that respects a closure property inherited from the monoidal structure.
\end{rem}

\begin{slab}[Injectives associated to Serre $\otimes$-ideals]
As illustrated in \cite[\S3]{BKSrum}, associated to any Serre $\otimes$-ideal $\mc{S}$ is a localisation $Q_{\mc{S}}\colon\Mod{\K}\to \Mod{\K}/\rlim\mc{S}$, which is a strong symmetric monoidal exact functor that admits a right adjoint $R_{\mc{S}}\colon\Mod{\K}/\rlim\mc{S}\to\Mod{\K}$. Letting $\alpha\colon Q_{\mc{S}}\yy\1\hookrightarrow\mbb{E}_{\mc{S}}$ denote the injective hull of $Q_{\mc{S}}\yy \1\in\Mod{\K}/\rlim\mc{S}$, applying $R_{\mc{S}}$ gives an injective object $J_{\mc{S}}\coloneqq R_{\mc{S}}\mbb{E}_{\mc{S}}\in\Mod{\K}$. Since $R_\mc{S}$ is fully faithful, we have $Q_\mc{S}J_\mc{S} = \mbb{E}_\mc{S}$. We note that $J_{\mc{S}}\in\scr{D}(\mc{S})$: if $f\in\mc{S}$ then \[\Hom(f,J_{\mc{S}})=\Hom(f,R_{\mc{S}}\mbb{E}_{\mc{S}})\simeq \Hom(Q_{\mc{S}}(f),\mbb{E}_{\mc{S}})=0.\]
\end{slab}

The Serre $\otimes$-ideal $\mc{S}$ can then be recovered from the injective object $J_\mc{S}$. This was proved in the rigidly-compactly generated setting in \cite[Theorem 3.5]{BKSrum}.

\begin{lem}\label{lem:recoverkernel}
Let $\mc{S}$ be a Serre $\otimes$-ideal of $\mod{\K}$. Then $\rlim\mc{S}=\msf{Ker}(-\otimes J_{\mc{S}})$, and as such $\mc{S} = \msf{Ker}(- \otimes J_\mc{S}) \cap \mod{\K}$.
\end{lem}

\begin{proof}
The second claim follows from the first by intersecting with $\mod{\K}$ so it suffices to prove the first. Since most of the proof of \cite[Theorem 3.5]{BKSrum} carries over to this setting, we only give a sketch. For the implication $\msf{Ker}(-\otimes J_{\mc{S}})\subseteq\rlim\mc{S}$, we note that if $M\in\Mod{\K}$, then there is a morphism $f\colon X\to Y$ in $\Flat{\K}$ such that $M\simeq \msf{Im}(f)$. Indeed, we can consider the composition $Y\twoheadrightarrow M\hookrightarrow X$ where $Y\in\msf{Proj}(\K)$ is projective, and $X$ is the injective envelope of $M$, which is flat by the discussion in \cref{slab:modules}. From this point the proof of the implication is the same as the corresponding at \cite[Theorem 3.5]{BKSrum}.

For the implication $\rlim\mc{S}\subseteq\msf{Ker}(-\otimes J_{\mc{S}})$, it is sufficient, since $J_{\mc{S}}$ is flat, to restrict to finitely presented objects. At this point, one adapts \cite[Proposition 3.3]{BKSrum}. Define a map $\eta\colon\yy \1\to J_{\mc{S}}$ via the composition
\[
\begin{tikzcd}
\yy \1 \arrow[r, "\t{unit}"]& R_{\mc{S}}Q_{\mc{S}}\yy \1 \arrow[r, "R_{\mc{S}}\alpha"] & J_{\mc{S}}= R_{\mc{S}}\mbb{E}_{\mc{S}}
\end{tikzcd}
\]
and consider the exact sequence
\[
0\to \msf{ker}(\eta) \to \yy \1\xrightarrow{\eta} J_{\mc{S}}\to \msf{coker}(\eta)\to 0
\]
in $\Mod{\K}$, which is the replacement for the triangle $\Delta$ in \cite[Proposition 3.3]{BKSrum} and its proof. From this point, the adaptation of said proof is straightforward, and the conclusion follows.
\end{proof}

\begin{slab}[Weak rings]\label{weakrings}
    An object $R \in \Mod{\K}$ is a \emph{weak ring} if there is a map $\eta\colon \yy\1 \to R$ such that $R \otimes \eta\colon R \to R \otimes R$ is a split monomorphism. Note that if $R$ is a weak ring then $R \otimes R$ is non-zero. Given any Serre $\otimes$-ideal $\mc{S}$ of $\mod{\K}$, the associated injective object $J_\mc{S}$ is a weak ring. Indeed, we have a map $\eta\colon \yy\1 \to J_\mc{S}$ defined as the composite
    \[\yy\1 \xrightarrow{\mrm{unit}} R_\mc{S}Q_\mc{S}\yy\1 \xhookrightarrow{R_\mc{S}(\alpha)} J_\mc{S}.\] Since $J_\mc{S}$ is injective and hence flat, the functor $Q_\mc{S}(J_\mc{S} \otimes -) \simeq \mbb{E}_\mc{S} \otimes Q_\mc{S}(-)$ is exact so we obtain a monomorphism
    \[\mbb{E}_\mc{S} \xrightarrow{\sim} \mbb{E}_\mc{S} \otimes Q_{\mc{S}}\yy\1 \hookrightarrow \mbb{E}_\mc{S} \otimes \mbb{E}_\mc{S}.\] This is a monomorphism out of an injective and hence splits. Now
    \[\Hom(-,J_\mc{S}) = \Hom(-, R_\mc{S}(\mbb{E}_\mc{S})) = \Hom(Q_\mc{S}(-),\mbb{E}_\mc{S})\] so this retraction lifts to a map $J_\mc{S} \otimes J_\mc{S} \to J_\mc{S}$ splitting $J_\mc{S} \otimes \eta$ as required.
\end{slab}

\begin{slab}[The homological spectrum]\label{homprimes}
Following \cite{BalmerNil}, a \emph{homological prime} of $\K$ is a maximal proper Serre $\otimes$-ideal of $\mod{\K}$. The set of homological primes of $\K$ is denoted $\hspec{\K}$, and it is endowed with a topology with a basis of closed sets given by 
\[
\supp^{h}(k)\coloneqq \{\mc{B}\in\hspec{\K}:\yy k\not \in \mc{B}\}.
\]
as $k$ ranges over $\K$.
\end{slab}

\begin{rem}
    One can conceptualise the relationship between the homological spectrum and definable subcategories as follows. Write $\msf{Zg}^{\otimes}(\K)$ for the topological space whose points are the isomorphism classes of indecomposable pure injective objects of $\Flat{\K}$, and whose closed sets are those of the form $\mc{D}\cap\msf{pinj}(\Flat{\K})$, where $\mc{D}\subseteq\Flat{\K}$ is a $\otimes$-closed definable subcategory. There is a homeomorphism $\Spch(\K) \simeq \closed{\K}^{\msf{GZ}}$ where the latter consists of the closed points of the Kolmogorov quotient of $\msf{Zg}^\otimes(\K)$ equipped with the Gabriel-Zariski topology, see \cite{BWHomSp}. In loc. cit. this is proved when $\K = \T^\c$ for some rigidly-compactly generated tt-category $\T$. One can easily adapt the arguments to give the above claim since there is a homeomorphism $\msf{Zg}^\otimes(\T) \simeq \msf{Zg}^\otimes(\msf{Flat}(\T^\c))$.
\end{rem}

\section{Definable functoriality of the homological spectrum}
\begin{slab}[Setup]\label{setup}
Let $\K$ and $\L$ be essentially small rigid tt-categories, and let $F\colon\Flat{\K}\to\Flat{\L}$ be a definable functor. We assume that $F$ has a left adjoint $\Lambda\colon \Flat{\L} \to \Flat{\K}$, and that this adjunction satisfies the projection formula. Explicitly, this means that $F$ is lax monoidal and the canonical map \[FX \otimes Y \xrightarrow{\sim} F(X \otimes \Lambda Y)\] is an isomorphism for all $X \in \Flat{\K}$ and $Y \in \Flat{\L}$.
\end{slab}

\begin{rem}\label{rem:barLambda}
    Given a definable functor $F\colon \Flat{\K} \to \Flat{\L}$, by \cite[Corollary 10.5]{Krauselfp} this functor extends to a definable functor $\overline{F}\colon \Mod{\K} \to \Mod{\L}$ which fits into an adjoint pair
\[
\begin{tikzcd}
\Mod{\K} \arrow[r, shift right = 0.5ex,swap,  "\bar{F}"] \arrow[r, leftarrow, shift left = 0.5ex, "\bar{\Lambda}"]& \Mod{\L}
\end{tikzcd}
\]
where $\bar{\Lambda}$ is exact and preserves finitely presented objects. As such, one can rephrase the requirement that $F$ has a left adjoint $\Lambda$ to the equivalent statement that $\bar{\Lambda}$ preserves flat objects.
\end{rem}

\begin{rem}
    The setup of starting with a definable functor $F\colon \Flat{\K} \to \Flat{\L}$ may appear odd at first glance. However, this is done to allow one to work with essentially small tt-categories. If $\K = \T^\c$ and $\L = \U^\c$ are the full subcategories of compact objects in some rigidly-compactly generated tt-categories $\T$ and $\U$, then one can instead start with the data of a definable functor $\T \to \U$ of triangulated categories, that is, a (not necessarily triangulated) functor which preserves coproducts and filtered homology colimits, or equivalently, coproducts, products, and pure triangles. Indeed, such a definable functor of triangulated categories yields a definable functor $\Flat{\K} \to \Flat{\L}$ by \cite[Theorem B]{BWdef}. 
\end{rem}

\begin{ex}\label{ex:geometric}
The canonical example satisfying the above setup comes from a geometric functor $f^*\colon\U \to \T$, that is, a coproduct preserving strong monoidal triangulated functor between rigidly-compactly generated tt-categories. The functor $f^*$ has a right adjoint $f_*\colon \T \to \U$, which is a definable functor of triangulated categories, see \cite[5.3]{BWdef}. By \cite[Theorems 4.16 and 4.21]{BWdef}, there is a definable functor $\widehat{f_*}\colon \Flat{\T^\c} \to \Flat{\U^\c}$ which extends to a functor $\overline{f_*}\colon \Mod{\T^\c} \to \Mod{\U^\c}$ which has a left adjoint $\overline{f^*}$. Since $\overline{f^*}$ commutes with the Yoneda embeddings, it preserves flats, and hence $\widehat{f_*}$ has a left adjoint $\widehat{f^*}$ given by the restriction of $\overline{f^*}$. The projection formula holds for the adjunction $(f^*,f_*)$ by \cite[Proposition 2.15]{BDS}, and since both $\widehat{f^*}$ and $\widehat{f_*}$ preserve direct limits, it follows that the projection formula holds for $(\widehat{f^*}, \widehat{f_*})$.
\end{ex}

\begin{ex}\label{ex:standard}
Suppose that $f\colon\scr{L}\to\scr{K}$ is an exact functor between essentially small rigid tt-categories. By identifying $\scr{K}$ with $\msf{proj}(\scr{K})$, and likewise for $\scr{L}$, we extend $f$ along direct limits to obtain a functor $\Lambda\colon\Flat{\scr{L}}\to\Flat{\scr{K}}$ which preserves direct limits. We may then apply \cite[Theorem 6.7]{Krauselfp} to obtain an adjoint pair
\[
\begin{tikzcd}
\Flat{\scr{K}} \arrow[r, shift right = 0.5ex, swap, "F"] \arrow[r, leftarrow, shift left = 0.5ex, "\Lambda"]& \Flat{\scr{L}}
\end{tikzcd}
\]
where $F$ is definable. Whenever $f$ is strong monoidal, so is $\Lambda$, and thus the adjunction satisfies the projection formula: the projection formula holds for the rigid objects by \cite[Proposition 3.2]{FHM}, but since the rigid objects are the finitely presented projectives we may take colimits to deduce it on all objects. Therefore in this case, we are in the situation of \cref{setup}. We note that this is the `small' version of \cref{ex:geometric}; indeed, we could deduce the previous example from this by taking $f$ to be the restriction of $f^*$ to compact objects.
\end{ex}

We now state our main result, and will assemble a proof throughout the rest of the section.

\begin{thm}\label{thm:functoriality}
Let $F\colon\Flat{\K}\to\Flat{\L}$ be a definable functor satisfying the hypotheses of \cref{setup}. Then the assignment
\[
\mc{B}\mapsto \msf{ker}(-\otimes FJ_{\mc{B}})\cap\mod{\L}
\]
defines a continuous map $\hspec{F}\colon\hspec{\K}\to\hspec{\L}$.
\end{thm}

\begin{chunk}\label{chunk:imagedefinable}
    Given a definable subcategory $\mc{D} \subseteq \Flat{\K}$ and a definable functor $F\colon \Flat{\K} \to \Flat{\L}$ we write $\msf{pure}(F\mc{D})$ for the closure of $F\mc{D}$ under pure subobjects. Equivalently, this is the smallest definable subcategory of $\Flat{\L}$ containing $F\mc{D}$, see for example after Corollary 13.4 in \cite{dac}. 
\end{chunk}

We say that a $\otimes$-closed definable subcategory is \emph{simple} if it is non-zero, and contains no proper non-zero $\otimes$-closed definable subcategory.
\begin{lem}\label{lem:imagesimple}
Assume the setup of \cref{setup}. Let $\mc{D}\subseteq\Flat{\K}$ be a simple $\otimes$-closed definable subcategory. Then $\msf{pure}(F\mc{D})$ is also a simple $\otimes$-closed definable subcategory of $\Flat{\L}$.
\end{lem}

\begin{proof}
Let us first prove that $\msf{pure}(F\mc{D})$ is a $\otimes$-closed definable subcategory of $\Flat{\L}$. The fact that it is definable is standard as recalled in \cref{chunk:imagedefinable}, so we must show that if $l\in\L$ and $X\in\msf{pure}(F\mc{D})$, then $\yy l\otimes X\in\msf{pure}(F\mc{D})$. As $X\in\msf{pure}(F\mc{D})$ there is a pure monomorphism $X\to FD$ for some $D\in\mc{D}$. The functor $\yy l \otimes -$ is definable so the induced map $\yy l\otimes X\to \yy l\otimes FD \simeq F(D \otimes \Lambda\yy l)$ is also a pure monomorphism. Consequently, as $D\otimes \Lambda \yy l\in\mc{D}$ we have $\yy l\otimes X \in \msf{pure}(F\mc{D})$ as claimed.

Let us now suppose that there is a $\otimes$-closed definable subcategory $0\neq\mc{E}\subseteq\msf{pure}(F\mc{D})$ in $\Flat{\L}$; we must show that $\mc{E} = \msf{pure}(F\mc{D})$. Consider the Serre $\otimes$-ideal $\scr{S}(\mc{E})\subseteq\mod{\L}$ and the associated injective object $J_{\scr{S}(\mc{E})}$. As $J_{\scr{S}(\mc{E})}\in\scr{D}(\scr{S}(\mc{E})) = \mc{E}$ by \cref{homprimes}, it follows that there is a pure (in fact split) monomorphism $\alpha\colon J_{\scr{S}(\mc{E})}\to FD$ for some $D\in\mc{D}$.

Consider the object $\Lambda J_{\scr{S}(\mc{E})}\otimes D\in\Flat{\K}$: this makes sense since $J_{\scr{S}(\mc{E})}$ is injective hence flat. Observe that $\Lambda J_{\scr{S}(\mc{E})}\otimes D\neq 0$: if it were zero, we would have $J_{\scr{S}(\mc{E})} \otimes FD \simeq F(\Lambda J_{\scr{S}(\mc{E})}\otimes D)=0$, and this cannot happen, since $J_{\scr{S}(\mc{E})}\otimes J_{\scr{S}(\mc{E})}$ is a summand of $J_{\scr{S}(\mc{E})}\otimes FD$ by the previous paragraph, and the former is non-zero as $J_{\scr{S}(\mc{E})}$ is a weak ring object, see \cref{weakrings}. 

As such, $\Lambda J_{\scr{S}(\mc{E})}\otimes D$ is a non-zero object in $\mc{D}$, and as $\mc{D}$ was assumed to be simple, we must have that $\msf{Def}^{\otimes}( \Lambda J_{\scr{S}(\mc{E})}\otimes D)=\mc{D}$. In other words, by \cref{slab:monoidalfundamental} we have
\[
\msf{Def}(\yy k\otimes \Lambda J_{\scr{S}(\mc{E})}\otimes D:k\in\K) = \mc{D},
\]
and as $F$ is a definable functor, we deduce that
\[
\msf{pure}(F\mc{D}) = \msf{Def}(F(\yy k\otimes \Lambda J_{\scr{S}(\mc{E})}\otimes D):k\in\K) = \msf{Def}(J_{\scr{S}(\mc{E})} \otimes F(\yy k \otimes D): k \in \K)
\]
where the latter equality follows from the projection formula.
Since $\mc{E}$ is a $\otimes$-closed definable subcategory and $J_{\scr{S}(\mc{E})} \in \mc{E}$, we have that $J_{\scr{S}(\mc{E})}\otimes F(\yy k\otimes D)\in\mc{E}$. Consequently, $\msf{pure}(F\mc{D}) \subseteq \mc{E}$, so that $\msf{pure}(F\mc{D}) = \mc{E}$, as required. Hence $\msf{pure}(F\mc{D})$ is simple as claimed.
\end{proof}

\begin{cor}\label{cor:defimage}
Assume the setup of \cref{setup} and let $\mc{B}\in\hspec{\K}$. Then $\msf{Def}^{\otimes}(FJ_{\mc{B}})=\msf{pure}(F\msf{Def}^{\otimes}(J_{\mc{B}}))$.
\end{cor}
\begin{proof}
Firstly we note that the object $FJ_\mc{B}$ is non-zero: since $J_\mc{B}$ is a weak ring and $F$ is lax monoidal, $F J_\mc{B}$ is also a weak ring and hence non-zero. The $\otimes$-closed definable subcategory $\scr{D}(\mc{B})\subseteq\Flat{\K}$ is simple since the fundamental correspondence is order-reversing, see \cref{fundamentalcorrespondence} and \cref{slab:monoidalfundamental}. As $J_\mc{B} \in \scr{D}(\mc{B})$ by \cref{homprimes} and $\scr{D}(\mc{B})$ is simple, there is an equality 
\begin{equation}\label{eq:identityDofhomprime}
    \scr{D}(\mc{B})=\msf{Def}^{\otimes}(J_{\mc{B}}).
\end{equation}
As $\msf{Def}^\otimes(J_\mc{B})$ is simple, by \cref{lem:imagesimple}, so is $\msf{pure}(F\msf{Def}^{\otimes}(J_{\mc{B}}))$. It is clear that $FJ_{\mc{B}}\in \msf{pure}(F\msf{Def}^{\otimes}(J_{\mc{B}}))$, so there is an inclusion $\msf{Def}^{\otimes}(FJ_{\mc{B}})\subseteq \msf{pure}(F\msf{Def}^{\otimes}(J_{\mc{B}}))$, but the simplicity of the latter, and the former being non-zero, means the two classes are equal.
\end{proof}

We are now in a position to prove the main theorem.
\begin{proof}[Proof of \cref{thm:functoriality}]
Let $\mc{B}\in\hspec{\K}$ and consider the simple $\otimes$-closed definable subcategory $\scr{D}(\mc{B})\subseteq\Flat{\K}$. By \cref{lem:imagesimple}, \cref{cor:defimage} and \cref{eq:identityDofhomprime}, we have that $\msf{Def}^{\otimes}(FJ_{\mc{B}})$ is a simple $\otimes$-closed definable subcategory of $\Flat{\L}$, and thus by the monoidal fundamental correspondence there is a homological prime $\mc{S}\coloneqq \scr{S}(\msf{Def}^{\otimes}(FJ_{\mc{B}}))\in\hspec{\L}$. We define $\hspec{F}(\mc{B})\coloneqq\mc{S}$, and it remains to show that this has the form as claimed in the statement.

Observe that $\msf{Def}^{\otimes}(FJ_{\mc{B}})=\scr{D}(\mc{S})=\msf{Def}^{\otimes}(J_{\mc{S}})$ where the latter equality holds as in \cref{eq:identityDofhomprime}. For any $f\in\mod{\L}$, the functor $f\otimes -\colon\Mod{\L}\to\Mod{\L}$ is definable, so it follows that $f\otimes J_{\mc{S}}=0$ if and only if $f\otimes FJ_{\mc{B}}=0$. Thus \begin{equation}\label{eq:image}\mc{S}=\msf{ker}(-\otimes J_{\mc{S}})\cap\mod{\L}=\msf{ker}(-\otimes FJ_{\mc{B}})\cap\mod{\L}\end{equation} by \cref{lem:recoverkernel} as desired. 

For continuity, we show that for any $l\in\L$, the preimage of the basic closed set $\supp^{h}(l)$ under $\hspec{F}$ is closed. By the above, for any $\mc{B} \in \Spch(\K)$ there is an equality $\msf{Def}^{\otimes}(FJ_{\mc{B}})=\msf{Def}^{\otimes}(J_{\msf{\hspec{F}(\mc{B})}})$. Therefore, since $\yy l\otimes -$ is a definable functor, we have $\yy l\otimes FJ_{\mc{B}}=0$ if and only if $\yy l\otimes J_{\hspec{F}(\mc{B})}=0$, and
\[
\hspec{F}^{-1}(\supp^{h}(l))=\{\mc{B}\in\hspec{\K}:\yy l\otimes FJ_{\mc{B}}\neq 0\}. 
\]
By the projection formula, we have $\yy l \otimes FJ_{\mc{B}} \simeq F(\Lambda\yy l \otimes J_{\mc{B}})$.
Since $J_{\mc{B}}$ is a weak ring object, the functor $F$ is conservative on objects of the form $\Lambda(f)\otimes J_{\mc{B}}$ for any $f\in\Flat{\L}$ by \cite[Remark 13.12]{bchs}, and as such we have $\yy l\otimes FJ_{\mc{B}}\neq 0 $ if and only if $\Lambda \yy l\otimes J_{\mc{B}}\neq 0$. Being the left adjoint to a direct limit preserving functor, $\Lambda$ preserves finitely presented objects. Therefore $\Lambda \yy l\simeq \yy k$ for some $k\in\K$, and thus 
$\hspec{F}^{-1}(\supp^{h}(l))=\supp^{h}(k)$, which is a closed set, as required.
\end{proof} 

We note that one can also give an alternative description of the continuous map $\Spch(F)$. Recall the functor $\overline{\Lambda}\colon \Mod{\L} \to \Mod{\K}$ from \cref{rem:barLambda}.
\begin{cor}\label{cor:alternativeform}
    Let $F\colon \Flat{\K} \to \Flat{\L}$ be a definable functor as in \cref{setup}. Then the continuous map $\Spch(F)\colon \Spch(\K) \to \Spch(\L)$ of \cref{thm:functoriality} takes the form
    \[\Spch(F)(\mc{B}) = \overline{\Lambda}^{-1}(\mc{B}) \cap \mod{\L}\] for all $\mc{B} \in \Spch(\K).$
\end{cor}
\begin{proof}
    We have
    \begin{align*}
    \msf{ker}(-\otimes FJ_{\mc{B}})\cap\mod{\L} &= \scr{S}(\msf{Def}^\otimes(FJ_\mc{B})) & \text{by \cref{eq:image}} \\
    &= \scr{S}(\msf{pure}(F\msf{Def}^\otimes(J_\mc{B}))) &\text{by \cref{cor:defimage}} \\
    &= \overline{\Lambda}^{-1}(\scr{S}(\msf{Def}^\otimes(J_\mc{B}))) \cap \mod{\L} &\text{by \cite[Corollary 4.13]{BWdef}} \\
    &= \overline{\Lambda}^{-1}(\mc{B}) \cap \mod{\L} &\text{by \cref{eq:identityDofhomprime}}
    \end{align*}
    as required.
\end{proof}

\begin{slab}[A comparison with Balmer's functoriality]
There is already a notion of functoriality for the homological spectrum and we now show that the functoriality of \cref{thm:functoriality} extends this result. Let $f^*\colon\U \to \T$ be a geometric functor between rigidly-compactly generated tt-categories. By \cref{ex:geometric} this fits into the setting of \cref{setup}, so we may apply \cref{thm:functoriality} together with \cref{cor:alternativeform} to recover the following result of Balmer~\cite[Theorem 5.10]{Balmerhomsupp}:
\begin{cor}\label{cor:generalise}
    Let $f^*\colon \U \to \T$ be a geometric functor between rigidly-compactly generated tt-categories. There is a continuous map $\Spch(\T^\c) \to \Spch(\U^\c)$ given by the assignment $\mc{B} \mapsto (\overline{f^*})^{-1}(\mc{B}) \cap \mod{\U^\c}$. \qed
\end{cor}
\end{slab}

\begin{slab}[Induced functoriality on the Balmer spectrum]
As a consequence of \cref{thm:functoriality}, we also obtain functoriality of the Balmer spectrum under the same assumptions:
\begin{thm}\label{thm:Balmerspectrumfunctoriality}
    Let $F\colon \Flat{\K} \to \Flat{\L}$ be a definable functor satisfying the hypotheses of \cref{setup}. There is a continuous map on Balmer spectra
    \[\Spc(F)\colon \Spc(\K) \to \Spc(\L)\]
    induced by taking the Kolmogorov quotient.
\end{thm}
\begin{proof}
    By \cite[Lemma 4.2]{BHSComp}, the map $\yy^{-1}\colon \Spch(\K) \to \Spc(\K)$ exhibits $\Spc(\K)$ as the Kolmogorov quotient of $\Spch(\K)$. Since the Kolmogorov quotient is functorial, applying it to the continuous map $\Spch(\K) \to \Spch(\L)$ of \cref{thm:functoriality} yields the claim.
\end{proof}
\end{slab}

\begin{rem}
    In the case when $F$ arises from an exact, strong monoidal functor $f\colon \L \to \K$ as in \cref{ex:standard}, it follows from \cref{cor:alternativeform} that the continuous map $\Spc(\K) \to \Spc(\L)$ of the previous theorem is given by $\p \mapsto f^{-1}(\p)$. In particular, we recover the usual functoriality of the Balmer spectrum.
\end{rem}

\bibliographystyle{abbrv}
\bibliography{references.bib}
\end{document}